\documentclass{article}

\usepackage{amsmath,amsthm,amssymb,amsfonts, graphicx}
\usepackage{graphics}
\usepackage{authblk}
\usepackage{upgreek}
\usepackage{amsrefs,hyperref}
\usepackage{amsthm}
\usepackage{thmtools}
\usepackage{cleveref}
\usepackage{dsfont}
\theoremstyle{plain}
\newtheorem{theorem}{Theorem}[section]
\newtheorem{corollary}[theorem]{Corollary}

\newtheorem{lemma}[theorem]{Lemma}
\newtheorem{proposition}[theorem]{Proposition}

\theoremstyle{definition}
\newtheorem{definition}[theorem]{Definition}
\theoremstyle{remark}
\newtheorem{remark}[theorem]{Remark}
\numberwithin{equation}{section}
\newcommand{\diff}{\mathop{}\!\mathrm{d}}

\DeclareMathOperator{\sgn}{sgn}

\title{Singular, finite-time $L^2$ attractors for odd, smooth solutions of Burgers equation on the torus}
\author[1]{Evan Miller}
\affil[1]{University of Maine,

evan.miller1@maine.edu}

\begin{document}
\maketitle

\begin{abstract}
    In this paper, we show that the positive multiples of a particular function $F$---which is singular with a jump discontinuity at the origin---are finite-time global attractors in $L^2$ for generic odd, smooth solutions of the one dimensional inviscid Burgers equation. Furthermore, the identity that leads to this result provides to an alternative proof of finite-time blowup for the fractal Burgers equation in the supercritical range $0<\alpha<\frac{1}{2}$. This proof is based on lower bounds on a Lyapunov functional given by the inner product of the solution with the global attractor $F$. We will also show that this property holds for a broader class of odd functions that are strictly increasing on $(0,\pi)$.
\end{abstract}

\section{Introduction}

The Burgers equation is a simple, but fundamentally important nonlinear PDE; it is the simplest nonlinear PDE to develop shock-type singularities. It is straightforward to compute these singular solutions semi-explicitly using the method of characteristics. The 1D viscous Burgers equation is given by
\begin{equation} \label{BurgersEqn}
    \partial_t u+u\partial_x u=0,
\end{equation}
and this equation can also be expressed in ``divergence form" as
\begin{equation}
    \partial_t u+\frac{1}{2}\partial_x u^2=0. 
\end{equation}
Burgers introduced and analyzed this equation, along with the viscous Burgers equation 
\begin{equation}
    \partial_t u-\nu \partial_x^2u+u\partial_x u=0,
\end{equation}
with viscosity $\nu>0$ in \cites{Burgers1939,Burgers1954a,Burgers1954b}, although the viscous Burgers equation first appeared in the work of Bateman \cite{Bateman}.
Hopf \cite{Hopf1950} and Cole \cite{Cole1951} independently proved global-in-time existence for smooth solutions of the viscous Burgers equation by taking a transformation (now known as the Hopf-Cole transformation) that allows the solution to be expressed explicitly in terms of a solution of the heat equation.

In this paper, we will consider the finite-time blowup problem for the inviscid Burgers equation and for the Burgers equation with fractional dissipation.
Consider the $2\pi$-periodic function
\begin{align}
    F(x)
    &=
    \begin{cases}
    -\pi+x, & 0<x\leq \pi \\
    0, & x=0 \\
    \pi+x, & -\pi\leq x<0
    \end{cases} \\
    &=
    -2\sum_{n=1}^\infty
    \frac{\sin(nx)}{n}.
\end{align}
We will show that for odd functions, there is an interesting identity for the inner product of $F$ with the Burgers nonlinearity. For all $2\pi$-periodic, odd $u\in \dot{H}^1\left(\mathbb{T}\right)$,
\begin{equation} \label{KeyIdIntro}
    \left<u\partial_xu,F\right>
    =
    \int_{-\pi}^\pi u(x)\partial_xu(x)F(x)\diff x
    =
    -\frac{1}{2}
    \|u\|_{L^2}^2.
\end{equation}
This immediately implies that any positive multiple of $F$ is an attractor for odd solutions of Burgers equation.

\begin{theorem} \label{GeneralAttractorThmIntro}
Suppose $u\in C\left([0,T_{max});
    C^1\left(\mathbb{T}\right)\right)$
    is an odd solution of Burgers equation and that $r>0$.
    Then for all $0\leq t<T_{max}$,
    \begin{equation}
    \frac{\diff}{\diff t}\|u(\cdot,t)-rF\|_{L^2}^2
    =
    -r\|u\|_{L^2}^2,
    \end{equation}
    and therefore
    \begin{equation}
    \|u(\cdot,t)-rF\|_{L^2}^2
    =
    \left\|u^0-rF\right\|_{L^2}^2
    -r\left\|u^0\right\|_{L^2}^2t.
    \end{equation}
    Note that this implies the solution blows up in finite-time with
    \begin{equation}
    T_{max}
    \leq
    \frac{\left\|u^0-rF\right\|_{L^2}^2}
    {r\left\|u^0\right\|_{L^2}^2}.
    \end{equation}
\end{theorem}

    If we minimize over $r>0$ to see which multiple $rF$ is reached in the shortest time, we find that it is the multiple with the same $L^2$ norm as the solution $u$, with
    \begin{equation}
    \inf_{r>0}
    \frac{\left\|u^0-rF\right\|_{L^2}^2}
    {r\left\|u^0\right\|_{L^2}^2}
    =
    \frac{\left\|u^0-r_0F\right\|_{L^2}^2}
    {r_0\left\|u^0\right\|_{L^2}^2},
    \end{equation}
    where 
    \begin{equation}
    r_0=\frac{\left\|u^0\right\|_{L^2}}{\|F\|_{L^2}}.
    \end{equation}

\begin{theorem} \label{AttractorThmIntro}
    Let $\Phi$ be the $L^2$ normalization of $F$,
\begin{align}
    \Phi(x)
    &=
    \frac{F}{\|F\|_{L^2}} \\
    &=
    \frac{\sqrt{3}}{\sqrt{2\pi^3}}
    \begin{cases}
    -\pi+x, & 0<x\leq \pi \\
    0, & x=0 \\
    \pi+x, & -\pi\leq x<0
    \end{cases}.
\end{align}
    Suppose $u\in C\left([0,T_{max});
    C^1\left(\mathbb{T}\right)\right)$
    is an odd solution of the inviscid Burgers equation.
    Then for all $0\leq t<T_{max}$,
    \begin{equation}
    \Big\|u(\cdot,t)-
    \left\|u^0\right\|_{L^2}\Phi\Big\|_{L^2}^2
    =
    \Big\|u^0-
    \left\|u^0\right\|_{L^2}\Phi\Big\|_{L^2}^2
    -\frac{\sqrt{3}}{\sqrt{2\pi^3}}
    \left\|u^0\right\|_{L^2}^3 t.
    \end{equation}
    Note in particular that this implies that 
    \begin{equation}
    T_{max}
    \leq 
    \left(\frac{\sqrt{2\pi^3}}{\sqrt{3}}
    \right)
    \frac{\Big\|u^0-
    \left\|u^0\right\|_{L^2}
    \Phi\Big\|_{L^2}^2}
    {\left\|u^0\right\|_{L^2}^3}.
    \end{equation}
\end{theorem}

\begin{remark}
    Finite-time blowup for the inviscid Burgers equation is, of course, as classical a blowup result as exists in nonlinear PDEs. It is nonetheless interesting to see that, for the odd subspace, the singular function $F$ is an attractor in the natural energy space $L^2$.
\end{remark}

\begin{remark}
    Note that while we have shown that $\left\|u^0\right\|_{L^2} \Phi$ is a finite-time attractor for any odd, smooth solution of Burgers equation, it is not guaranteed that this function is reached before blowup time, and in general it will not be. We can show that the $L^2$ distance between the solution and this function shrinks, which establishes a lower bound on the blowup time, but in general blowup will occur earlier with
    \begin{equation}
    T_{max}
    <
    \left(\frac{\sqrt{2\pi^3}}{\sqrt{3}}
    \right)
    \frac{\Big\|u^0-
    \left\|u^0\right\|_{L^2}
    \Phi\Big\|_{L^2}^2}
    {\left\|u^0\right\|_{L^2}^3},
    \end{equation}
    and in this case
    \begin{equation}
    \lim_{t\to T_{max}}u(x,t)
    \neq 
    \left\|u^0\right\|_{L^2}\Phi(x).
    \end{equation}
\end{remark}

\begin{remark}
    We also note the striking similarity between the the attractor function $F$ for solutions of the inviscid Burgers equation and the ``N-wave" which has long been observed as the long-time dynamics of the viscous Burgers equation with small viscosity. For example, see Figure 2b in \cite{Burgers1939}, which is a regularized version of $F$. This N-wave behaviour has been observed as a the longtime dynamics of the viscous Burgers equation in a wide range of contexts as long as the viscosity is sufficiently small that the dissipation can be treated as a perturbation of the nonlinearity and not vice versa. This is quite natural; the dynamics of the inviscid attractor will still be present in the viscous equation, but with the jump discontinuity smoothed out to a region with a very negative spatial derivative $\partial_x u$.
\end{remark}

The 1D Burgers equation can also be seen as a sort of simplified model for the incompressible Euler equation, because it also involves the self-advection of velocity, albeit in a simpler setting, and with no incompressibility condition. The viscous Burgers equation can be used as a simplified model of the Navier--Stokes equation, and can be expressed as
\begin{equation}
    \partial_t u-\nu\Delta u+u\partial_x u=0,
\end{equation}
where $-\Delta=-\partial_x^2$. To further understand the competing effects of the nonlinearity and dissipation, the fractal Burgers equation was introduced, involving a fractional Laplacian dissipation:
\begin{equation}
    \partial_t u +\nu(-\Delta)^\alpha u
    +u\partial_x u=0,
\end{equation}
$0<\alpha<1$.
This allows the study of the interplay between the nonlinearity and the dissipation.
We will sometimes refer to \eqref{BurgersEqn} as the inviscid Burgers equation for clarity.

The wellposedness theory and finite-time blowup for the fractal Burgers equation is well established. Kiselev, Nazarov, and Shterenberg proved that there are smooth solutions that blowup in finite-time for all $0<\alpha<\frac{1}{2}$, the supercritical case, and that there are global smooth solutions whenever $\alpha\geq \frac{1}{2}$. The proof of finite-time blowup when $\alpha<\frac{1}{2}$ involves some very detailed time-splitting arguments. Independently and around the same time, Alibaud, Droniou, and Vovelle provided finite-time blowup when $0<\alpha<\frac{1}{2}$, for solutions on the whole real line, also making use of a time splitting argument \cite{Alibaud}. Shortly after, Dong, Du, and Li \cite{DongDuLi} provided another proof of finite-time blowup in the supercritical case based on a Lyapunov function with singular weights at the origin.
Very recently, Pasqualotto and Oh proved a finite-time blowup result for certain dissipative and dispersive perturbations of Burgers equation \cite{OhPasqualotto}.
This result includes supercritical blowup for the fractal Burgers equation as a special case, as well as the Whitham and fractional KdV equations.

Taking the inner product of a solution of the fractal Burgers equation with $F$, the global attractor in the inviscid case, gives a new proof of finite-time blowup in the range $0<\alpha<\frac{1}{2}$.

\begin{theorem} \label{BurgersBlowupThmIntro}
Suppose $u\in C\left([0,T_{max}),\dot{H}^s\right), s>\frac{3}{2}-2\alpha$ is a solution of the fractal Burgers equation with $\alpha<\frac{1}{2}$ and with odd initial data satisfying
    \begin{equation}
    L_0^3>16\pi^3 C_\alpha^2
    \left\|u^0\right\|_{L^2}^2 \nu,
    \end{equation}
where
\begin{align}
    L_0&:=\int_{-\pi}^\pi u^0(x)F(x)\diff x \\
    C_\alpha&:=\sqrt{2\pi}\left(\sum_{n=1}^\infty 
    \frac{1}{n^{2(1-\alpha)}}\right)^\frac{1}{2}.
\end{align}
Then the solution of the fractal Burgers equation blows up in finite-time
\begin{equation}
T_{max}<\frac{4\pi^3}{L_0}.
\end{equation}
\end{theorem}

For clarity, we provide an example of an initial data where there is finite-time blowup. In fact, for all $0<\alpha<\frac{1}{2}$, blowup is sufficiently generic that taking the initial data consisting of only the lowest order sine wave guarantees blowup subject only to a sign and largeness condition.

\begin{corollary} \label{BlowupBurgersExampleCorIntro}
     Suppose $u\in C\left([0,T_{max}),
     C^\infty\left(\mathbb{T}\right)\right)$ 
     is the solution of the fractal Burgers equation with $\alpha<\frac{1}{2}$ and $\nu>0$, with initial data
    \begin{equation}
    u^0(x)=-R \sin(x),
    \end{equation}
    satisfying
    \begin{equation}
    \frac{R}{\nu}>8\pi^2\sum_{n=1}^\infty 
    \frac{1}{n^{2(1-\alpha)}}.
    \end{equation}
    Then $u(\cdot,t)$ blows up in finite-time 
    \begin{equation}
    T_{max}< \frac{2\pi^2}{R}.
    \end{equation}
\end{corollary}

\begin{remark}
    We can see very clearly in this result the fact that the critical threshold is for finite-time blowup on the dissipation is $\alpha<\frac{1}{2}$. If we attempted to push the result out to $\alpha=\frac{1}{2}$,
    we run into the divergent series 
    \begin{equation}
    \sum_{n=1}^\infty \frac{1}{n}=+\infty,
    \end{equation}
    Note that the quantity $\frac{R}{\nu}$, serves as an analog of Reynolds number, comparing the magnitude of the flow to the strength of the viscosity, 
    so the theorem guarantees finite-time blowup with lowest order sine wave as initial data at large Reynolds number subject only to a sign condition.
    The closer we get to $\alpha=\frac{1}{2}$, the larger the Reynolds number must be to guarantee the solution with this particular initial data to blowup in finite-time, because
    \begin{equation}
    \lim_{\alpha\to\frac{1}{2}^-}
    \sum_{n=1}^\infty 
    \frac{1}{n^{2(1-\alpha)}}
    =+\infty.
    \end{equation}
\end{remark}

\begin{remark}
    Note that when working on the whole space, the solution set the 1D fractal Burgers equation is preserved for all $\lambda>0$ by the transformation
    \begin{equation}
    u^\lambda(x,t)=\lambda^{2\alpha-1}u(\lambda x, \lambda^{2\alpha} t).
    \end{equation}
    It is a classical result that if $u$ is a smooth solution of the 1D inviscid Burgers equation, then for all $1\leq q\leq +\infty$ and for all $0<t<T_{max}$,
    \begin{equation}
    \|u(\cdot,t)\|_{L^q}=\left\|u^0\right\|_{L^q}.
    \end{equation}
    This fact can be used to find a priori bounds on all order $L^p$ norms for the fractal Burgers equation. The bound in $L^\infty$ is the strongest, and it can be seen that this bound is supercritical when $\alpha<\frac{1}{2}$, critical when $\alpha=\frac{1}{2}$, and subcritical when $\alpha>\frac{1}{2}$.
    The fact that the $L^\infty$ norm remains bounded then leads to global regularity when $\alpha\geq \frac{1}{2}$. See \cite{KiselevBurgers} for details.
\end{remark}

\begin{remark}
Note that this has a significant similarity with the Lyapunov functional used by Dong, Du, and Li in \cite{DongDuLi}. One of the major advances of \cite{DongDuLi} was that this paper does not require an oddness assumption on the initial data; however, in the odd case specifically there are some similarities with the results in this paper. For odd solutions of the fractal Burgers equation, they prove finite-time blowup when $0<\alpha<\frac{1}{2}$, so long as for some $0<\delta<1-2\alpha$
\begin{equation}
    \int_{-\pi}^\pi
    u^0(x) G_\delta(x)
    >
    C_{\alpha,\delta}
    \nu^\frac{1}{2}
    \left\|u^0\right\|_{L^\infty}^\frac{1}{2},
\end{equation}
where
\begin{equation}
    G_\delta(x)=
    \sgn(x)\left(
    \pi^{-\delta}-|x|^{-\delta}\right),
\end{equation}
and $C_{\alpha,\delta}>0$ is a constant depending only on $\alpha$ and $\delta$.
Note that the sign of $G_\delta$ has been flipped from \cite{DongDuLi}, because they use a different convention with the nonlinearity having the opposite sign.

In fact, it is even remarked in \cite{DongDuLi} that the function $F$ could also be used as a Lyapunov functional, using the methods in that paper. However, the precise identity for the inner product of the Burgers nonlinearity with the Lyapunov functional $F$ means that the proof of finite-time blowup for the fractal Burgers equation in this paper does not rely on a priori bounds on the $L^\infty$ norm, but rather uses a priori bounds from the energy equality. This may be useful in other contexts, because the energy equality for the fractal Burgers equation has a direct analogue in the hypodissipative Navier--Stokes equation, while the a priori bound on the $L^\infty$ norm does not.
\end{remark}

Finally, we will consider the dynamics of odd solutions of the fractal Burgers equation in Fourier space.
Any odd solution of the fractal Burgers equation on the torus can be expressed in the form
\begin{equation} 
        u(x,t)=-2\sum_{n=1}^\infty 
        \psi_n(t) \sin(nx).
    \end{equation}
Once the equation is expressed in this way, the dynamics can be entirely described by the following infinite system of ODEs.
    \begin{equation} \label{IntroODE} 
    \partial_t \psi_n
    =-\nu n^{2\alpha}\psi_n
    +\frac{n}{2}\sum_{j=1}^{n-1} \psi_j\psi_{n-j}
    -n\sum_{k=1}^\infty \psi_k\psi_{k+n}.
    \end{equation}
Using Parseval's identity, the Lyapunov functional can be naturally expressed in terms of the Fourier coefficients as follows:
\begin{equation} \label{IntroAltLyapunov}
L(t)=\int_{-\pi}^\pi F(x)u(x,t)\diff x
= 4\pi \sum_{n=1}^\infty \frac{\psi_n}{n},
\end{equation}
and indeed we first considered this Lyapunov functional in terms of this Fourier series expansion.
It may seem that the function $F$ is rather arbitrary, but we can see that the $\frac{1}{n}$ Fourier sine series expansion has a very natural connection to the nonlinearity in \eqref{IntroODE}, because it cancels the factor of $n$ in front of each of the nonlinear terms, and it was precisely that structure from the nonlinearity which motivated the introduction of the Lyapunov functional, which we first expressed in terms of the Fourier coefficients as in \eqref{IntroAltLyapunov}, and only later in terms of integrating against $F$.

Note that \Cref{GeneralAttractorThmIntro,BurgersBlowupThmIntro} can be generalized from the specific function $F$ to a wide range of odd functions that are differentiable and strictly increasing on $(0,\pi)$.

\begin{theorem} \label{VeryGeneralAttractorThmIntro}
    Suppose that $H\in L^\infty\left(\mathbb{T}\right)$ is odd, differentiable on $(0,\pi)$, and that 
    \begin{equation}
    m:= \inf_{0<x<\pi} H'(x)
    >0.
    \end{equation}
    Suppose that $u\in C\left([0,T_{max});
    C^1\left(\mathbb{T}\right)\right)$
    is an odd solution of Burgers equation.
    Then for all $0\leq t<T_{max}$,
    \begin{equation}
    \|u(\cdot,t)-H\|_{L^2}^2
    \leq 
    \left\|u^0-H\right\|_{L^2}^2
    -mt.
    \end{equation}
    Note that this implies the solution blows up in finite-time with
    \begin{equation}
    T_{max}\leq\frac{\left\|u^0-H\right\|_{L^2}^2}{m}.
    \end{equation}
\end{theorem}

\begin{theorem}  \label{VeryGeneralBlowupThmIntro}
 Suppose that $H\in \dot{H}^\alpha\left(\mathbb{T}\right) \cap L^\infty\left(\mathbb{T}\right), 0<\alpha<\frac{1}{2}$ is odd, differentiable on $(0,\pi)$, and that 
    \begin{equation}
    m:= \inf_{0<x<\pi} H'(x)
    >0.
    \end{equation}    
    Suppose $u\in C\left([0,T_{max}),\dot{H}^s\right),  s>\frac{3}{2}-2\alpha$ is a solution of the fractal Burgers equation with fractional dissipation $\alpha$ and with odd initial data satisfying
    \begin{equation}
    \Tilde{L}_0^3
    >\frac{12}{m}
    \|H\|_{\dot{H}^\alpha}^2
    \|H\|_{L^2}^2
    \left\|u^0\right\|_{L^2}^2\nu,
    \end{equation}
    where
    \begin{equation}
    \Tilde{L}_0
    =
    \int_{-\pi}^\pi
    H(x) u^0(x)\diff x.
    \end{equation}
Then the solution of the fractal Burgers equation blows up in finite-time
\begin{equation}
T_{max}<\frac{6\|H\|_{L^2}^2}{m\Tilde{L}_0}.
\end{equation}
\end{theorem}

\begin{remark}
    Note that while \Cref{VeryGeneralBlowupThmIntro} shows that blowup is highly generic for the fractal Burgers equation with $\alpha<\frac{1}{2}$ when the viscosity is sufficiently small compared with the initial data, we can also see why $\alpha<\frac{1}{2}$ is the critical threshold. Note that the requirement that $H$ is odd and strictly increasing on $(0,\pi)$ automatically implies that $H$ is also strictly increasing on $(-\pi,0)$. Because $H$ is $2\pi$-periodic, this means that there must be a jump discontinuity (with the left hand limit greater than the right hand limit) at either $x=0$ or at $x=\pm \pi$. 
    This immediately implies that $H\notin \dot{H}^s\left(\mathbb{T}\right)$ for all $s>\frac{1}{2}$, as Sobolev embedding would then guarantee continuity. We require $H\in \dot{H}^\alpha$ for the proof, and so the argument clearly cannot be generalized to $\alpha>\frac{1}{2}$. It also cannot be generalized to $\alpha=\frac{1}{2}$, but this case is more subtle, as this is the exact endpoint of Sobolev embedding, and relies on the fact that there is a $\delta$ distribution in the distributional derivative of $H$ which just fails to be in $\dot{H}^\frac{1}{2}$.
    
    In the case of $F$, we have a jump discontinuity at $x=0$. For an example of the later case, consider the function $2\pi$ periodic function given on $[-\pi,\pi]$ by  
    \begin{equation}
    H(x)=\begin{cases}
        x, & -\pi<x<\pi \\
        0, & x=\pm\pi.
    \end{cases}
    \end{equation}
    It is straightforward to check that $H$ satisfies the hypotheses of \Cref{VeryGeneralAttractorThmIntro,VeryGeneralBlowupThmIntro}, and in fact that $H$ is just a translate of $F$ with $H(x)=F(x+\pi)$.
\end{remark}

\subsection{Outline and structure}

We begin this paper, in \Cref{AttractorSection}, by showing that positive multiples of $F$ are $L^2$ attractors for the inviscid Burgers equation, proving \eqref{KeyIdIntro} and \Cref{GeneralAttractorThmIntro,AttractorThmIntro}. 
Then in \Cref{BlowupSection}, we will use the key identity from \eqref{KeyIdIntro} to prove finite-time blowup when $0<\alpha<\frac{1}{2}$, proving \Cref{BurgersBlowupThmIntro} and \Cref{BlowupBurgersExampleCorIntro}.
In \Cref{FourierSection}, we will consider the dynamics of odd solutions of the Burgers equation in Fourier space.
Finally in \Cref{VeryGeneralSection}, we will consider a general class of functions that are $L^2$ attractors for the inviscid Burgers equation, of which $F$ is only one example, and show that these functions can also be used as Lyapunov functions to prove finite-time blowup for the fractal Burgers equation, proving \Cref{VeryGeneralAttractorThmIntro,VeryGeneralBlowupThmIntro}.
We structure the paper in this way because the $L^2$ attractor result is the most novel aspect of the paper, but we also note that this structure is exactly reversed from the order in which the results were discovered.

All of the results in this paper were first discovered on the Fourier space side, and only later were expressed in terms of $F$. While \Cref{FourierSection} is not strictly required---the results in this section have already been expressed in physical space in \Cref{AttractorSection,BlowupSection}---it is nonetheless included as motivation for considering how the the function $F$ interacts with solutions of Burgers equation. Without it, $F$ would appear to drop from the sky rather arbitrarily.

\subsection{Definitions and preliminaries}

Kiselev, Nazarov, and Shterenberg proved the following local wellposedness result in \cite{KiselevBurgers} that will provide the setting for our study of smooth solutions of the fractal Burgers equation.
\begin{theorem} \label{KiselevThm}
    Suppose $0<\alpha\leq \frac{1}{2}$, 
    and suppose $u^0\in \dot{H}^s\left(\mathbb{T}\right), s>\frac{3}{2}-2\alpha$.
    Then there exists some positive time $T=T(\alpha,s,\|u\|_{\dot{H}^s})>0$, where there is a unique solution
    $u\in C\left([0,T];\dot{H}^s\right)\cap L^2\left([0,T];\dot{H}^{s+\alpha}\right)$.
    Note that the uniqueness is among solutions with this regularity, not necessarily among all possible weak solutions.
    Furthermore, we have higher regularity for all positive times, $u\in C^\infty\left((0,T],\mathbb{T}\right)$.
    We also have the energy equality
    \begin{equation} \label{EnergyEquality}
    \frac{1}{2}\|u(\cdot,t)\|_{L^2}^2
    +\nu\int_0^t\|u(\cdot,\tau)\|_{\dot{H}^\alpha}^2
    =\frac{1}{2}\left\|u^0\right\|_{L^2}^2.
    \end{equation}
\end{theorem}

\begin{remark}
    Note that this solution is a weak solution in the sense that it satisfies,
    \begin{equation}
    \int_0^T\int_{-\pi}^\pi 
    \left(-u \partial_t w -\frac{1}{2}u^2\partial_x w
    +u (-\Delta)^{\alpha}w\right)(x,t)
    \diff x \diff t,
    \end{equation}
    for all test functions $w\in C_c^\infty\left(
    [0,T]\times\mathbb{T}\right)$,
    but that higher regularity guarantees that $u$ is also a classical solution for all positive times up until the blowup time. Note that because $\alpha\leq \frac{1}{2}$, we are in the quasilinear setting, and so the a solution cannot be obtained by the Banach fixed point argument using the fractional heat semigroup. Nonetheless---while the solution is obtained as the weak limit of a mollified problem---as in the Euler equation, there are bounds on the growth of the Sobolev norms that are uniform in the mollification due to the structure of the nonlinearity.
    This is enough to guarantee uniqueness because of stability in $\dot{H}^{s-1}$;
    see \cite{KiselevBurgers} for more details. 
\end{remark}

We will also need to consider solutions of the inviscid Burgers equation \eqref{BurgersEqn} on the torus. It is very classical that smooth solutions exist locally-in-time for any data in $C^1\left(\mathbb{T}\right)$, but see in particular \cites{Burgers1954a,Burgers1954b} for the following result.

\begin{theorem}
    Suppose $u^0\in C^1\left(\mathbb{T}\right)$. Then there exists a unique, local strong solution of Burgers equation 
    $u\in C\left([0,T_{max});
    C^1\left(\mathbb{T}\right)\right)$.
    If there exists $x\in\mathbb{T}$ such that $\partial_x u^0(x)<0$, then 
    \begin{equation}
    T_{max}=\frac{1}
    {-\min_{x\in\mathbb{T}}\partial_x u^0(x)}
    <+\infty;
    \end{equation}
    and if for all $x\in\mathbb{T}, 
    \partial_x u^0(x)\geq 0$,
    then $T_{max}=+\infty$.
    Note that for solutions on the torus, this implies that $T_{max}<+\infty$ for an solution with non-constant initial data.
    Furthermore, for all $1\leq q\leq +\infty$ and for all $0\leq t<T_{max}$,
    \begin{equation}
    \|u(\cdot,t)\|_{L^q}
    =
    \left\|u^0\right\|_{L^q}.
    \end{equation}
    Note that the case $q=2$ corresponds to the energy equality for the incompressible Euler equation.
\end{theorem}

We will consider $2\pi$ periodic functions $u:\mathbb{T}\to \mathbb{R}$ satisfying $u(x)=u(x+2\pi)$, and we will represent the 1D torus by $\mathbb{T}=[-\pi,\pi]$.
We will define a number of function spaces of periodic functions that will be essential to our analysis.
\begin{definition}
    For all $u\in L^2$, we take the Fourier Transform to be
    \begin{equation}
    \hat{u}(k)=
    \frac{1}{2\pi} \int_{-\pi}^\pi
    u(y)e^{-iky}\diff y,
    \end{equation}
    which gives the $L^2$ convergent Fourier series
    \begin{equation}
    u(x)=\sum_{k\in\mathbb{Z}}\hat{u}(k)e^{ikx}.
    \end{equation}
    For all $s\geq 0$, define the homogeneous Sobolev norm by
    \begin{equation}
    \|u\|_{\dot{H}^s}^2
    =
    2\pi \sum_{k\in\mathbb{Z}}
    k^{2s}|\hat{u}(k)|^2,
    \end{equation}
    and we will define the space $\dot{H}^s\left(\mathbb{T}\right)$ by
    \begin{equation}
    \dot{H}^s\left(\mathbb{T}\right)
    =
    \left\{u\in L^2
    \left(\mathbb{T}\right):
    \hat{u}(0)=0, \|u\|_{\dot{H}^s}<+\infty\right\}.
    \end{equation}
    \end{definition}
    
    Note that the condition $\hat{u}(0)=0$ is the mean-free condition
    \begin{equation}
    \int_{-\pi}^\pi u(x)\diff x=0.
    \end{equation}
    Taking the standard $L^q$ norm
    \begin{equation}
    \|u\|_{L^q}
    =\left(\int_{-\pi}^\pi
    |u(x)|^q \diff x\right)^\frac{1}{q},
    \end{equation}
    we can see from the Parseval's identity that 
    for all $u\in \dot{H}^s$,
    \begin{equation}
    \|u\|_{\dot{H}^0}=\|u\|_{L^2}.
    \end{equation}

    We will define the fractional Laplacian 
    $(-\Delta)^\alpha: \dot{H}^{2\alpha} \to L^2$, as a Fourier multiplier with
    \begin{equation}
    \widehat{(-\Delta)^\alpha u}(k)
    =
    k^{2\alpha}\hat{u}(k),
    \end{equation}
    and in particular when $\alpha=1$, we have $-\Delta=-\partial_x^2$.

\section{An attractor for the inviscid Burgers equation} \label{AttractorSection}

In this section, we will prove that positive multiples of $F$ are finite-time global attractors for odd solutions of the inviscid Burgers equation.
Recall $F\in L^2\left(\mathbb{T}\right)$ be the $2\pi$-periodic function defined 
on the interval $[-\pi,\pi]$ by
\begin{equation}
    F(x)
    =
    \begin{cases}
    -\pi+x, & 0<x\leq \pi \\
    0, & x=0 \\
    \pi+x, & -\pi\leq x<0.
    \end{cases}
\end{equation}
Note that $F(-\pi)=F(\pi)$ so the only discontinuity on $[-\pi,\pi]$ is at $x=0$, and more generally at $x\in 2\pi \mathbb{N}$.
We begin with our key identity.

\begin{lemma} \label{InnerProductLemma}
    Suppose $u\in \dot{H}^1\left(\mathbb{T}\right)$ is odd. Then 
    \begin{equation}
    \left<F,u\partial_x u\right>
    =
    -\frac{1}{2}\|u\|_{L^2}^2.
    \end{equation}
\end{lemma}

\begin{proof}
    First note that by H\"older's inequality,
    \begin{equation}
    \left|\left<F,u\partial_x u\right>\right|
    \leq 
    \|F\|_{L^\infty}
    \|\partial_x u\|_{L^2}
    \|u\|_{L^2}
    <
    +\infty,
    \end{equation}
    so the inner product is well defined.
    By hypothesis $u$ and $F$ are odd, and therefore $\partial_x u$ is even, $u \partial_x u$ is odd, and consequently $F u\partial_x u$ is even.
    This implies that
    \begin{align}
    \left<F,u\partial_x u\right>
    &=
    2\int_0^\pi F(x)
    u(x)\partial_xu(x)\diff x \\
    &=
    \int_0^\pi (x-\pi)\partial_x
    u^2(x)\diff x \\
    &=
    -\int_0^\pi u^2(x)\diff x \\
    &=
    -\frac{1}{2}\|u\|_{L^2}^2,
    \end{align}
    where the boundary terms vanish because
    $u(0)=0$ and $u(\pi)=0$ because of oddness (and $2\pi$ periodicity in the latter case).
\end{proof}

We can see that \Cref{InnerProductLemma} gives an immediate proof of \Cref{GeneralAttractorThmIntro}, which is restated for the reader's convenience.

\begin{theorem} \label{GeneralAttractorThm}
    Suppose $u\in C\left([0,T_{max});
    C^1\left(\mathbb{T}\right)\right)$
    is an odd solution of the inviscid Burgers equation
    and that $r>0$.
    Then for all $0\leq t<T_{max}$,
    \begin{equation}
    \|u(\cdot,t)-rF\|_{L^2}^2
    =
    \left\|u^0-rF\right\|_{L^2}^2
    -r\left\|u^0\right\|_{L^2}^2 t.
    \end{equation}
    Note in particular that this implies that
    \begin{equation}
    T_{max}
    \leq 
    \frac{\left\|u^0-rF\right\|_{L^2}^2}
    {r\left\|u^0\right\|_{L^2}^2}
    \end{equation}
\end{theorem}

\begin{proof}
    First observe that due to the energy equality
    \begin{align}
    \|u(\cdot,t)-rF\|_{L^2}^2
    &=
    \|u(\cdot,t)\|_{L^2}^2
    -2r\left<u(\cdot,t),F\right>
    +r^2\|F\|_{L^2}^2 \\
    &=
    \left\|u^0\right\|_{L^2}^2
    -2r\left<u(\cdot,t),F\right>
    +r^2\|F\|_{L^2}^2.
    \end{align}
    Applying \Cref{InnerProductLemma} and again making use of the energy equality, we find
    \begin{align}
    \frac{\diff}{\diff t}
    \|u(\cdot,t)-rF\|_{L^2}^2
    &=
    -2r \frac{\diff}{\diff t}
    \left<u(\cdot,t),F\right> \\
    &=
    2r\left<u\partial_x u,F\right> \\
    &=
    -r\|u(\cdot,t)\|_{L^2}^2 \\
    &=
    -r\left\|u^0\right\|_{L^2}^2.
    \end{align}
Integrating this differential equality completes the proof.
\end{proof}

We have shown that for all $r>0, rF$ is a finite-time $L^2$ attractor for every nonzero, odd solution of Burgers equation. It would seem surprising that the $rF$ could be an attractor for different values of $r$, as one function must be approached.
This confusion can be eliminated by computing
\begin{equation}
    \inf_{r>0}
    \frac{\left\|u^0-rF\right\|_{L^2}^2}
    {r\left\|u^0\right\|_{L^2}^2}.
\end{equation}

\begin{proposition} \label{InfProp}
    For all $u^0\in L^2\left(\mathbb{T}\right)$,
    \begin{equation}
    \inf_{r>0}
    \frac{\left\|u^0-rF\right\|_{L^2}^2}
    {r\left\|u^0\right\|_{L^2}^2}
    =
    \frac{\left\|u^0-r_0 F\right\|_{L^2}^2}
    {r_0\left\|u^0\right\|_{L^2}^2},
    \end{equation}
    where
    \begin{equation}
    r_0=\frac{\left\|u^0\right\|_{L^2}}
    {\|F\|_{L^2}}
    \end{equation}
\end{proposition}

\begin{proof}
    Let $g\in C^\infty\left(\mathbb{R}^+\right)$ be given by
    \begin{equation}
    g(r)=\frac{\left\|u^0-rF\right\|_{L^2}^2}
    {r\left\|u^0\right\|_{L^2}^2}.
    \end{equation}
    Observe that
    \begin{equation}
    g(r)=\frac{1}{r}
    -\frac{2\left<u^0,F\right>}
    {\left\|u^0\right\|_{L^2}^2}
    +\frac{r\|F\|_{L^2}^2}
    {\left\|u^0\right\|_{L^2}^2}.
    \end{equation}
    Therefore, for all $r>0$
    \begin{equation}
    g'(r)
    =
    -\frac{1}{r^2}
    +\frac{\|F\|_{L^2}^2}
    {\left\|u^0\right\|_{L^2}^2}.
    \end{equation}
    This implies that 
    \begin{equation}
    \inf_{r>0}g(r)
    =
    g(r_0),
    \end{equation}
    where
    \begin{equation}
    r_0=\frac{\left\|u^0\right\|_{L^2}}
    {\|F\|_{L^2}},
    \end{equation}
    and this completes the proof.
\end{proof}

\begin{remark}
    This helps to clarify \Cref{GeneralAttractorThm}. While for all $r>0$, the function $rF$ is an $L^2$ attractor for any odd solution of Burgers equation, we can see that the appropriate normalization is the attractor which would be reached in the shortest amount of time.
    It is straightforward to compute that
    \begin{equation}
    \|F\|_{L^2}^2=\frac{2}{3}\pi^3,
    \end{equation}
    and so the normalization in $L^2$ can be expressed
    \begin{equation}
    \Phi(x)
    =
    \frac{F}{\|F\|_{L^2}}
    =
    \frac{\sqrt{3}}{\sqrt{2\pi^3}}
    \begin{cases}
    -\pi+x, & 0<x\leq \pi \\
    0, & x=0 \\
    \pi+x, & -\pi\leq x<0
    \end{cases}.
\end{equation}
We can see from \Cref{InfProp} that the positive multiple of $F$ approached most rapidly is exactly the multiple with the same $L^2$ norm as the solution, namely $\left\|u^0\right\|_{L^2} \Phi$.
Plugging $r=r_0$ into \Cref{GeneralAttractorThm} and substituting in $\Phi$ for F, we find that 
for all $0\leq t<T_{max}$,
    \begin{equation}
    \Big\|u(\cdot,t)-
    \left\|u^0\right\|_{L^2}\Phi\Big\|_{L^2}^2
    =
    \Big\|u^0-
    \left\|u^0\right\|_{L^2}\Phi\Big\|_{L^2}^2
    -\frac{\sqrt{3}}{\sqrt{2\pi^3}}
    \left\|u^0\right\|_{L^2}^3 t,
    \end{equation}
    and in particular
    \begin{equation}
    T_{max}
    \leq 
    \left(\frac{\sqrt{2\pi^3}}{\sqrt{3}}
    \right)
    \frac{\Big\|u^0-
    \left\|u^0\right\|_{L^2}
    \Phi\Big\|_{L^2}^2}
    {\left\|u^0\right\|_{L^2}^3}.
    \end{equation}
This completes the proof of \Cref{AttractorThmIntro}.
\end{remark}

\section{Finite-time blowup} \label{BlowupSection}

In this section, we will prove finite-time blowup for the fractal Burgers equation whenever $\alpha<\frac{1}{2}$. The argument will be based on singular lower bounds for the Lyapunov functional
\begin{equation}
    L(t)=\int_{-\pi}^{\pi}u(x,t)F(x) \diff x.
\end{equation}
We begin by considering the Fourier series for $F$.

\begin{proposition}
The function $F\in L^\infty\left(\mathbb{T}\right)$ has the convergent Fourier series
\begin{equation}
F(x)
=
-2\sum_{n=1}^\infty
\frac{\sin(nx)}{n}
=
i\sum_{\substack{k\in\mathbb{N}\\ k\neq 0}}
\frac{e^{ikx}}{k},
\end{equation}
with convergence both pointwise and in $L^2$.
Note that this implies that
\begin{equation}
    \hat{F}(k)=\begin{cases}
        \frac{i}{k}, & k\neq 0 \\
        0, & k=0,
    \end{cases}
\end{equation}
and so for all $0\leq s<\frac{1}{2}$,
$F\in \dot{H}^s
\left(\mathbb{T}\right)$
\begin{equation}
\|F\|_{\dot{H}^s}^2=4\pi \sum_{n=1}^\infty \frac{1}{n^{2-2s}}
\end{equation}
\end{proposition}

\begin{proof}
Observe that
\begin{equation}
a_n=
\frac{2}{\pi}\int_0^\pi 
\sin(nx)(x-\pi)\diff x
=
-\frac{2}{n},
\end{equation}
and the claim follows from classical results in Fourier analysis because $F\in L^2$ and $F$ is piecewise $C^1$.
\end{proof}

\begin{proposition} \label{LyapunovProp}
    Suppose $u\in C\left([0,T_{max}),\dot{H}^s\right), s>\frac{3}{2}-2\alpha$ is an odd solution of the fractal Burgers equation.
     Then for all $0<t<T_{max}$,
     \begin{equation}
    \frac{\diff}{\diff t}
    L(t)
    =
    -\nu \left<\left(-\Delta\right)^\frac{\alpha}{2}u,
    \left(-\Delta\right)^\frac{\alpha}{2}F\right>
    +\frac{1}{2}
    \|u\|_{L^2}^2.
     \end{equation}
     Furthermore, for all $0<t<T_{max}$,
     \begin{equation}
     \frac{\diff L}{\diff t}
     \geq 
     -\sqrt{2}C_\alpha\nu\|u\|_{\dot{H}^\alpha}
     + L^2,
     \end{equation}
     where 
     \begin{equation}
    C_\alpha=\sqrt{2\pi}\left(\sum_{n=1}^\infty 
    \frac{1}{n^{2(1-\alpha)}}\right)^\frac{1}{2}
     \end{equation}
\end{proposition}

\begin{proof}
    Differentiating the Lyapunov functional we find that
    \begin{align}
    \frac{\diff}{\diff t}
    L(t)
    &=
    -\frac{\nu}{4\pi}
    \left<(-\Delta)^\alpha u,F\right>
    -
    \left<F,u\partial_xu\right> \\
    &=
    -\nu \left<\left(-\Delta\right)^\frac{\alpha}{2}u,
    \left(-\Delta\right)^\frac{\alpha}{2}F\right>
    +\frac{1}{2}
    \|u\|_{L^2}^2,
    \end{align}
    where we have applied \Cref{InnerProductLemma} and used the fact that 
    $(-\Delta)^\frac{\alpha}{2}$ is self adjoint.

    Next apply H\"older's inequality to conclude that 
    \begin{align}
    \left<\left(-\Delta\right)^\frac{\alpha}{2}u,
    \left(-\Delta\right)^\frac{\alpha}{2}F\right>
    &\leq 
    \left\|\left(-\Delta\right)^\frac{\alpha}{2}F\right\|_{L^2}
    \left\|\left(-\Delta\right)^\frac{\alpha}{2}u\right\|_{L^2} \\
    &=
    \sqrt{4\pi}
    \left(\sum_{n=1}^\infty \frac{1}{n^{2(1-\alpha)}}\right)^\frac{1}{2}
    \|u\|_{\dot{H}^\alpha},
    \end{align}
    and that 
    \begin{align}
    |L|
    &\leq 
    \|F\|_{L^2}\|u\|_{L^2} \\
    &\leq
    \frac{\sqrt{2\pi^3}}{\sqrt{3}} \|u\|_{L^2}.
    \end{align}
    Therefore, we can see that 
    \begin{equation}
    \|u\|_{L^2}^2\geq \frac{3}{2\pi^3}L^2,
    \end{equation}
    and so
    \begin{equation}
    \frac{\diff L}{\diff t}
    \geq 
    -2\sqrt{\pi} \left(\sum_{n=1}^\infty 
    \frac{1}{n^{2(1-\alpha)}}\right)^\frac{1}{2} \nu
    \|u\|_{\dot{H}^\alpha}
    +\frac{3}{4\pi^3}L^2,
    \end{equation}
    and this completes the proof.
\end{proof}

\begin{proposition} \label{ODEprop}
    Suppose a scalar function $y\in C\left([0,T_{max})\right) \cap C^1\left((0,T_{max})\right)$ satisfies the differential inequality
    \begin{equation}
    \frac{\diff y}{\diff t}
    \geq -f +\kappa y^2,
    \end{equation}
    for all $0<t<T_{max}$, where $\kappa>0$ is a constant independent of time, and $f\in L^1\left([0,T_{max})\right)$ is a nonnegative integrable function satisfying
    \begin{equation}
    \int_0^t f(\tau)\diff \tau 
    \leq Mt^\frac{1}{2},
    \end{equation}
    for some $M>0$ independent of $t$,
    and $y(0)=y_0>0$.
    Then for all $0<t<\min\left(T_{max},
    \frac{y_0^2}{M^2}\right)$,
    \begin{equation}
    y(t)\geq
    \left(\frac{1}{y_0}-\kappa t
    +\frac{Mt^\frac{1}{2}}{(y_0-Mt^\frac{1}{2})^2}
    \right)^{-1}.
    \end{equation}
\end{proposition}

\begin{proof}
    We begin by observing that for all $0<t<T_{max}$
    \begin{equation}
    \frac{\diff y}{\diff t}
    \geq -f,
    \end{equation}
    and so
    \begin{align}
    y(t) 
    &\geq 
    y_0 -\int_0^t f(\tau)\diff\tau \\
    &\geq 
    y_0 -Mt^\frac{1}{2}.
    \end{align}
    Now compute that
    \begin{align}
    \frac{\diff}{\diff t}
    \left(-\frac{1}{y(t)}\right)
    &=
    \left(\frac{1}{y^2}\right)
    \frac{\diff y}{\diff t} \\
    &\geq
    \kappa -\frac{f}{y^2}.
    \end{align}
    Integrating this differential inequality from $0$ to $t$,
    we find that
    \begin{equation}
    \frac{1}{y_0}-\frac{1}{y(t)}
    \geq
    \kappa t-\int_0^t
    \frac{f(\tau)}{y(\tau)^2}\diff\tau.
    \end{equation}
    Rearranging this inequality, we find that for all $0<t<T_{max}$,
    \begin{equation} \label{Bound1}
    \frac{1}{y(t)}\leq 
    \frac{1}{y_0}-\kappa t
    +\int_0^t \frac{f(\tau)}{y(\tau)^2}\diff\tau.
    \end{equation}
    
    We know that for all $0<\tau<t$,
    \begin{equation}
    y(\tau)\geq y_0-M\tau^\frac{1}{2}
    \geq y_0- Mt^\frac{1}{2},
    \end{equation}
    and so for all $0<\tau<t<\min\left(T_{max},
    \frac{y_0^2}{M^2}\right)$
    \begin{align}
    \int_0^t \frac{f(\tau)}{y(\tau)^2}\diff\tau
    &\leq 
    \frac{1}{(y_0-Mt^\frac{1}{2})^2}
    \int_0^t f(\tau)\diff\tau \\
    &\leq 
    \frac{Mt^\frac{1}{2}}{(y_0-Mt^\frac{1}{2})^2}.
    \end{align}
    Plugging in this bound into \eqref{Bound1}, we find that, 
    for all $0<\tau<t<\min\left(T_{max},
    \frac{y_0^2}{M^2}\right)$,
    \begin{equation}
    y(t)\geq
    \left(\frac{1}{y_0}-\kappa t
    +\frac{Mt^\frac{1}{2}}{(y_0-Mt^\frac{1}{2})^2}
    \right)^{-1}.
    \end{equation}
\end{proof}

\begin{lemma} \label{ODElemma}
    Again suppose a scalar function $y\in C\left([0,T_{max})\right) \cap C^1\left((0,T_{max})\right)$ satisfies the differential inequality
    \begin{equation}
    \frac{\diff y}{\diff t}
    \geq -f +\kappa y^2,
    \end{equation}
    for all $0<t<T_{max}$, where $\kappa>0$ is a constant independent of time, and $f\in L^1\left([0,T_{max})\right)$ is a nonnegative integrable function satisfying
    \begin{equation}
    \int_0^t f(\tau)\diff \tau 
    \leq Mt^\frac{1}{2},
    \end{equation}
    for some $M>0$ independent of $t$,
    and $y(0)=y_0>0$.
    Further suppose that 
    \begin{equation}
    y_0^3 
    \geq
    \frac{12M^2}{\kappa}. 
    \end{equation}
    Then for all $0<t<\min\left(T_{max},
    \frac{y_0^2}{4M^2}\right)$,
    \begin{equation}
    y(t)>
    \left(\frac{3}{y_0}-\kappa t
    \right)^{-1},
    \end{equation}
    and in particular
    \begin{equation}
    T_{max}\leq \frac{3}{\kappa y_0}.
    \end{equation}
\end{lemma}

\begin{proof}
    Observe that for all $0<t<\frac{y_0^2}{4M^2}$,
    \begin{equation}
    Mt^\frac{1}{2}<\frac{y_0}{2},
    \end{equation}
    and consequently
    \begin{equation}
    y_0-Mt^\frac{1}{2}>\frac{y_0}{2},
    \end{equation}
    and therefore
    \begin{equation}
    \frac{Mt^\frac{1}{2}}{(y_0-Mt^\frac{1}{2})^2}
    <\frac{2}{y_0}.
    \end{equation}
    Applying \Cref{ODEprop}, we can see that
    for all $0<t<\min\left(T_{max},
    \frac{y_0^2}{4M^2}\right)$
    \begin{align}
    y(t)
    &\geq
    \left(\frac{1}{y_0}-\kappa t
    +\frac{Mt^\frac{1}{2}}{(y_0-Mt^\frac{1}{2})^2}
    \right)^{-1} \\
    &\>
    \left(\frac{3}{y_0}-\kappa t
    \right)^{-1}.
    \end{align}
    This bound immediately implies finite-time blowup with
    \begin{equation}
    T_{max}\leq\frac{3}{\kappa y_0},
    \end{equation}
    as long as 
    \begin{equation}
    \frac{3}{\kappa y_0}
    \leq 
    \frac{y_0^2}{4M^2},
    \end{equation}
    which follows from the condition
    \begin{equation}
    y_0^3 
    \geq
    \frac{12M^2}{\kappa}. 
    \end{equation}
\end{proof}

With this lemma available, we will now prove \Cref{BurgersBlowupThmIntro}, showing finite-time blowup for solutions of the fractal Burgers equation in the whole supercritical range $0<\alpha<\frac{1}{2}$.

\begin{theorem} \label{BurgersBlowupThm}
Suppose $u\in C\left([0,T_{max}),\dot{H}^s\right), s>\frac{3}{2}-2\alpha$ is a solution of the fractal Burgers equation with $\alpha<\frac{1}{2}$ and with odd initial data satisfying
    \begin{equation}
    L_0^3>16\pi^3 C_\alpha^2
    \left\|u^0\right\|_{L^2}^2 \nu
    \end{equation}
Then the solution of the fractal Burgers equation blows up in finite-time
\begin{equation}
T_{max}<\frac{4\pi^3}{L_0}.
\end{equation}
\end{theorem}

\begin{proof}
    We will use the energy equality and the bounds on the Lyapunov functional proven in \Cref{LyapunovProp}, to show that the differential inequality for the Lyapunov functional satisfies the conditions in \Cref{ODElemma}. Begin by letting
    \begin{equation}
    f(t)= \sqrt{2} C_\alpha \nu \|u(\cdot,t)\|_{\dot{H}^\alpha}.
    \end{equation}
    Applying H\"older's inequality to the functions $1$ and $\|u(\cdot,t)\|_{\dot{H}^\alpha}$, and using the bounds from the energy equality \eqref{EnergyEquality}, 
    we can see that for all $0<t<T_{max}$,
    \begin{align}
    \int_0^t\|u(\cdot,\tau)\|_{\dot{H}^\alpha}
    \diff\tau
    &\leq 
    \left(\int_0^t 1\diff\tau\right)^\frac{1}{2}
    \left(\int_0^t \|u(\cdot,\tau)\|_{\dot{H}^\alpha}^2\diff\tau
    \right)^\frac{1}{2}\\
    &=
    t^\frac{1}{2}
    \left(\int_0^t \|u(\cdot,\tau)\|_{\dot{H}^\alpha}^2\diff\tau
    \right)^\frac{1}{2}\\
    &\leq 
    \frac{t^\frac{1}{2}}{(2\nu)^\frac{1}{2}}
    \left\|u^0\right\|_{L^2}.
    \end{align}
    This implies that for all $0<t<T_{max}$
    \begin{equation}
    \int_0^t f(\tau)\diff\tau
    \leq 
    C_\alpha \nu^\frac{1}{2}
    \left\|u^0\right\|_{L^2} t^\frac{1}{2}.
    \end{equation}
    Applying \Cref{ODElemma} with
    \begin{align}
    M&= C_\alpha \nu^\frac{1}{2}
    \left\|u^0\right\|_{L^2} \\
    \kappa &= \frac{3}{4\pi^3},
    \end{align}
    we can see that for all $0<t<\min\left(
    T_{max},\frac{L_0^2}{4C_\alpha^2\left\|u^0\right\|_{L^2}^2
    \nu}\right)$,
    \begin{equation} \label{SingularLowerBound}
    L(t)>\left(\frac{3}{L_0}-\frac{3}{4\pi^3}t
    \right)^{-1},
    \end{equation}
    this completes the proof.
\end{proof}

\begin{remark}
    One interesting aspect of this proof is that the Lyapunov functional $L(t)$ cannot actually blow up in finite-time. 
    We know from \Cref{LyapunovProp} and the energy equality \eqref{EnergyEquality}, 
    this quantity is actually bounded by the energy with
    \begin{equation}
    L(t)\leq 
    \frac{\sqrt{2\pi^3}}{\sqrt{3}}\left\|u^0\right\|_{L^2}
    \end{equation}
    At first this appears to be a contradiction, but all this means is that the blowup has to occur before the singular lower bound \eqref{SingularLowerBound} would violate the energy equality, because the bound only holds so long as a strong solution exists.
    Even though $L(t)$ cannot actually blowup, the singular lower bounds imply finite-time blowup, because long time existence would lead to a contradiction. 
    
    It might still appear that there is a contradiction because a weak solution exists globally in time, so we can still continue past the blowup time considering the weak solution, and for a weak solution in $C_t^w L_x^2 \cap L^2 \dot{H}^\alpha_x$ the bound on $L(t)$ above by the energy will still hold.
    While this upper bound still holds, the singular lower bound in \Cref{BurgersBlowupThm} holds only for solutions with enough regularity, not for generic weak solutions. In particular, note that the integration by parts in \Cref{InnerProductLemma}, requires a certain degree of regularity, and that this aspect of the proof will not hold for generic weak solutions.
\end{remark}

We will now provide an explicit example of initial data where this condition is satisfied, proving \Cref{BlowupBurgersExampleCorIntro}, which is restated for the reader's convenience.

\begin{corollary}
    Suppose $u\in C\left([0,T_{max}),C^\infty\right)$ is the solution of the fractal Burgers equation with $\alpha<\frac{1}{2}$, with initial data
    \begin{equation}
    u^0(x)=-R \sin(x),
    \end{equation}
    with
    \begin{equation}
    \frac{R}{\nu} > 4\pi C_\alpha^2 .
    \end{equation}
    Then $u$ blows up in finite-time with
    \begin{equation}
    T_{max}< \frac{2\pi^2}{R}.
    \end{equation}
\end{corollary}

\begin{proof}
    Recalling the Fourier series for $F$, we can see that
    \begin{align}
    L_0
    &=
    -R\int_{-\pi}^\pi
    F(x)\sin(x) \\
    &=
    2\pi R
    \end{align}
    and furthermore that
    \begin{equation}
    \left\|u^0\right\|_{L^2}^2
    =
    \pi R^2.
    \end{equation}
    This implies that
    \begin{align}
    \frac{L_0^3}{\left\|u^0\right\|_{L^2}^2\nu}
    &=
    \frac{4\pi^2 R}{\nu} \\
    &>
    16\pi^3C_\alpha^2.
    \end{align}
    Applying \Cref{BurgersBlowupThm}, we can see that
    \begin{align}
    T_{max}
    &<
    \frac{4\pi^3}{L_0} \\
    &=
    \frac{2\pi^2}{R},
    \end{align}
    which completes the proof.
\end{proof}

\appendix

\section{The evolution equation in Fourier space} \label{FourierSection}

In this appendix, we will consider the dynamics of odd solutions of the fractal Burgers equation in Fourier space, reducing these dynamics to an infinite system of ODEs involving the coefficients of the Fourier sine series.
We begin by showing that the odd subspace is preserved and then showing that the dynamcics are described by the infinite system of ODEs given in \eqref{IntroODE}.

\begin{proposition} \label{OddProp}
    Suppose $u\in C\left([0,T_{max}),\dot{H}^s\right), s>\frac{3}{2}-2\alpha$ is a solution of fractal Burgers equation, and $u^0\in \dot{H}^s$ is odd. Then $u(x,t)$ is odd for all $0<t<T_{max}$. 
\end{proposition}

\begin{proof}
    Observe that if $u$ is a solution of the fractal Burgers equation, then $v(x,t)=-u(-x,t)$ is also a solution of the fractal Burgers equation, and that $u^0=v^0$. The result then follows from uniqueness.
\end{proof}

\begin{theorem} \label{EvoEqnThm}
    Suppose $u\in C\left([0,T_{max}),\dot{H}^s\right), s>\frac{3}{2}-2\alpha$ is an odd solution of the fractal Burgers equation. Then for all $0\leq t<T_{max}$,
    \begin{equation} \label{FourierSeries}
        u(x,t)=-2\sum_{n=1}^\infty 
        \psi_n(t) \sin(nx),
    \end{equation}
    and for all $0<t<T_{max}$,
    \begin{equation} \label{FourierODE}
    \partial_t \psi_n
    =-\nu n^{2\alpha}\psi_n
    +\frac{n}{2}\sum_{j=1}^{n-1} \psi_j\psi_{n-j}
    -n\sum_{k=1}^\infty \psi_k\psi_{k+n}.
    \end{equation}
\end{theorem}

\begin{proof}
    We know that every odd function must have a Fourier series representation of the form \eqref{FourierSeries}. The higher regularity results in \Cref{KiselevThm} implies that $\psi_n(t)$ must be infinitely differentiable and have faster than polynomial decay as $n\to \infty$.
    \Cref{OddProp} guarantees that the odd subspace is preserved by the dynamics, which means the entire dynamics of the fractal Burgers equation with odd initial data can be reduced to an infinite system of ODEs describing the Fourier coefficients $\psi_n$. It remains only to prove that the Fourier coefficients satisfy the evolution equation \eqref{FourierODE}.

    Begin by observing that 
    \begin{align}
    u^2(x,t)
    &=
    4\sum_{j,k=1}^\infty
    \psi_j(t)\psi_k(t) \sin(jx)\sin(kx) \\
    &=
    2\sum_{j,k=1}^\infty
    \psi_j\psi_k 
    \left(-\cos((j+k)x)+\cos((j-k)x)\right),
    \end{align}
    where we have used the fact that $\sin(a)\sin(b)=
    -\frac{1}{2}\cos(a+b)+\frac{1}{2}\cos(a-b)$.
    Because cosine is even, we can see that
    \begin{equation}
    \sum_{j<k}\psi_j\psi_k \cos((j-k)x)
    =
    \sum_{j>k}\psi_j\psi_k \cos((j-k)x),
    \end{equation}
    and so
    \begin{equation}
    \sum_{j,k=1}^\infty
    \psi_j\psi_k \cos((j-k)x)
    =
    2\sum_{j>k}\psi_j\psi_k \cos((j-k)x)
    +\sum_{j=1}^\infty \psi_j^2.
    \end{equation}
    This implies that 
    \begin{equation}
    u^2=
    -2\sum_{j,k=1}^\infty
    \psi_j\psi_k \cos((j+k)x)
    +4\sum_{j>k}\psi_j\psi_k \cos((j-k)x)
    +2\sum_{j=1}^\infty \psi_j^2.
    \end{equation}
    Letting $n=j+k$ in the first summation, and $n=j-k$, in the second summation, we find that
    \begin{equation}
    u^2
    =
    2\sum_{j=1}^\infty \psi_j^2
    -2\sum_{n=1}^\infty
    \left(\sum_{j=1}^{n-1} \psi_j\psi_{n-j}
    -2\sum_{k=1}^\infty \psi_k\psi_{k+n}\right)\cos(nx),
    \end{equation}
    which implies that
    \begin{equation}
    \partial_x\frac{1}{2}u^2
    =
    -2\sum_{n=1}^\infty
    \left(-\frac{n}{2}\sum_{j=1}^{n-1} \psi_j\psi_{n-j}
    +n\sum_{k=1}^\infty \psi_k\psi_{k+n}\right)\sin(nx).
    \end{equation}

    Finally, we observe that 
    \begin{align}
    \partial_t u(x,t)
    &=
    -2\sum_{n=1}^\infty 
    \partial_t\psi_n(t) \sin(nx) \\
    (-\Delta)^\alpha u
    &=
    -2\sum_{n=1}^\infty 
    n^{2\alpha}\psi_n \sin(nx).
    \end{align}
    Putting this all together implies that
    \begin{equation}
    \partial_t u +\nu(-\Delta)^\alpha u
    +\partial_x \frac{1}{2}u^2
    =
    -2\sum_{n=1}^\infty \left(
    \partial_t \psi_n
    +\nu n^{2\alpha}\psi_n
    -\frac{n}{2}\sum_{j=1}^{n-1} \psi_j\psi_{n-j}
    +n\sum_{k=1}^\infty \psi_k\psi_{k+n}
    \right) \sin(nx)
    =0,
    \end{equation}
    where we have written the fractal Burgers equation in divergence form.
    The uniqueness of the Fourier series then implies that for all $0<t<T_{max}$ and for all $n\in\mathbb{N}$,
    \begin{equation}
    \partial_t \psi_n
    +\nu n^{2\alpha}\psi_n
    -\frac{n}{2}\sum_{j=1}^{n-1} \psi_j\psi_{n-j}
    +n\sum_{k=1}^\infty \psi_k\psi_{k+n}
    =0.
    \end{equation}
    This completes the proof.
\end{proof}

It is necessary at this point to define Sobolev and $L^p$ spaces for sequences.

 \begin{definition}
 Consider $\psi:\mathbb{N}\to \mathbb{R}$.
    For all $s\geq 0$, define the $\dot{H}^s\left(\mathbb{T}\right)$ norm by
    \begin{equation}
    \|\psi\|_{\dot{H}^s}^2
    =
    \sum_{n=1}^\infty n^{2s}\psi_n^2,
    \end{equation}
    and we will say that $\psi\in \dot{H}^s\left(\mathbb{N}\right)$ if and only if $\|\psi\|_{\dot{H}^s}<+\infty$.
    For all $1\leq p<+\infty$, we define the $L^p\left(\mathbb{N}\right)$ by
    \begin{equation}
    \|\psi\|_{L^p}=
    \left(
    \sum_{n=1}^\infty |\psi_n|^p\right)^\frac{1}{p},
    \end{equation}
    and we will say that $\psi\in L^p\left(\mathbb{N}\right)$ if and only if $\|\psi\|_{L^p}<+\infty$.    
    \end{definition}

    \begin{remark}
    There is a direct relationship between $\dot{H}^s\left(\mathbb{T}\right)$ and 
    $\dot{H}^s\left(\mathbb{N}\right)$ for odd functions. Suppose that $u,v\in \dot{H}^s$ are odd. Then $u$ and $v$ must have unique Fourier sine series representations of the form
    \begin{align}
    u(x)&=
    -2\sum_{n=1}^\infty
    \psi_n \sin(nx) \\
    v(x)&=
    -2\sum_{n=1}^\infty
    \phi_n \sin(nx).
    \end{align}
    Furthermore, Parseval's identity can be expressed in this context as
    \begin{equation} 
    \|u\|_{\dot{H}^s
    \left(\mathbb{T}\right)}^2
    = 4\pi
    \|\psi\|_{\dot{H}^s
    \left(\mathbb{N}\right)}^2,
    \end{equation}
    and
    \begin{equation} \label{ParsevalID}
    \int_{-\pi}^\pi u(x)v(x) \diff x
    =
    4\pi\sum_{n=1}^\infty 
    \psi_n \phi_n
    \end{equation}
    Note that the normalization factor of $-2$ is chosen so that for all $n\in\mathbb{N}$,
    \begin{align}
    \hat{u}(n)&=i\psi_n \\
    \hat{u}(-n)&=-i\psi_n.
    \end{align}
    \end{remark}

\begin{remark}
    Both the conservation of energy and the estimates that lead to finite-time blowup for the Lyapunov functional given in \Cref{LyapunovProp} can be seen from the Fourier space side. In particular, the $\frac{-2}{n}$ Fourier sine series for $F$ was introduced to cancel the factor of $n$ in the infinite-system of ODEs \Cref{EvoEqnThm}.
\end{remark}

\begin{proposition} \label{EnergyProp}
     Suppose $u\in C\left([0,T_{max}),\dot{H}^s\right), s>\frac{3}{2}-2\alpha$ is an odd solution of the fractal Burgers equation.
     Then for all $0<t<T_{max}$,
     \begin{equation}
     \|\psi(\cdot,t)\|_{L^2}^2
     +2\nu\int_0^t 
     \|\psi(\cdot,\tau)\|_{\dot{H}^\alpha}^2\diff\tau
     =
     \left\|\psi^0\right\|_{L^2}^2.
     \end{equation} 
\end{proposition}

\begin{proof}
    Using the infinite system of ODEs from \Cref{EvoEqnThm}, we can see that for all $0<t<T_{max}$,
    \begin{equation}
    \frac{\diff}{\diff t}
    \|\psi(\cdot,t)\|_{L^2}^2
    =
    -2\nu \sum_{n=1^\infty}
    n^{2\alpha}\psi_n^2
    +\sum_{n=1}^\infty 
    \sum_{j=1}^{n-1}
    n\psi_n \psi_j\psi_{n-j}
    -2\sum_{n=1}^\infty\sum_{k=1}^\infty
    n\psi_n\psi_k\psi_{n+k}.
    \end{equation}
    Letting $m=n-j$, we find that
    \begin{align}
    \sum_{n=1}^\infty 
    \sum_{j=1}^{n-1}
    n\psi_n \psi_j\psi_{n-j}
    &=
    \sum_{j=1}^\infty 
    \sum_{m=1}^\infty
    (m+j)\psi_{m+j}\psi_m\psi_j \\
    &=
    2\sum_{j=1}^\infty 
    \sum_{m=1}^\infty
    m\psi_{m+j}\psi_m\psi_j,
    \end{align}
    due to the symmetry of the summand.
    Therefore we can conclude that
    \begin{equation}
    \sum_{n=1}^\infty 
    \sum_{j=1}^{n-1}
    n\psi_n \psi_j\psi_{n-j}
    =2\sum_{n=1}^\infty\sum_{k=1}^\infty
    n\psi_n\psi_k\psi_{n+k},
    \end{equation}
    and so for all $0<t<T_{max}$,
    \begin{equation}
    \frac{\diff}{\diff t}
    \|\psi(\cdot,t)\|_{L^2}^2
    =
    -2\nu \|\psi(\cdot,t)\|_{\dot{H}^\alpha}^2.
    \end{equation}
    Integrating this differential equation completes the proof.
\end{proof}

\begin{proposition}  \label{FourierLyapunovProp}
    Suppose $u\in C\left([0,T_{max}),\dot{H}^s\right), s>\frac{3}{2}-2\alpha$ is an odd solution of the fractal Burgers equation.
     Then for all $0<t<T_{max}$,
     \begin{equation}
    \frac{\diff}{\diff t}
    \sum_{n=1}^\infty \frac{\psi_n(t)}{n}
    =-\nu \sum_{n=1}^\infty \frac{\psi_n(t)}{n^{1-2\alpha}}
    +\frac{1}{2}\sum_{n=1}^\infty \psi_n(t)^2.
     \end{equation}
\end{proposition}

\begin{proof}
    Plugging into the evolution equation for the Fourier coefficients in \Cref{EvoEqnThm}, we find that 
    \begin{equation}
    \frac{\diff}{\diff t}
    \sum_{n=1}^\infty \frac{\psi_n(t)}{n}
    =
    -\nu \sum_{n=1}^\infty \frac{\psi_n(t)}{n^{1-2\alpha}}
    +\frac{1}{2}\sum_{n=1}^\infty \sum_{j=1}^{n-1} \psi_j\psi_{n-j}
    -\sum_{n=1}^\infty \sum_{j=1}^\infty \psi_j\psi_{j+n}
    \end{equation}
    Letting $m=n-j$, we find that
    \begin{equation}
    \sum_{n=1}^\infty \sum_{j=1}^{n-1} \psi_j\psi_{n-j}
    =\sum_{j,m=1}^\infty \psi_m \psi_j.
    \end{equation}
    Likewise setting $m=j+n$, we find that
    \begin{equation}
    \sum_{n=1}^\infty \sum_{j=1}^\infty \psi_j\psi_{j+n}
    =
    \sum_{1\leq j<m<\infty} \psi_j \psi_m.
    \end{equation}
    Note that by symmetry 
    \begin{equation}
    \sum_{1\leq j<m<\infty} \psi_j \psi_m
    =
    \sum_{1\leq m<j<\infty} \psi_j \psi_m
    =
    \frac{1}{2}\sum_{\substack{j,m \in \mathbb{N} \\ j\neq m}}
    \psi_j\psi_m,
    \end{equation}
    and so we can conclude that
    \begin{equation}
    \frac{1}{2}\sum_{n=1}^\infty \sum_{j=1}^{n-1} \psi_j\psi_{n-j}
    -\sum_{n=1}^\infty \sum_{j=1}^\infty \psi_j\psi_{j+n}
    =
    \frac{1}{2}\sum_{n=1}^\infty \psi_n^2,
    \end{equation}
    which completes the proof.
\end{proof}

\begin{proposition}
    Suppose $u\in C\left([0,T_{max}),\dot{H}^s\right), s>\frac{3}{2}-2\alpha$ is an odd solution of the fractal Burgers equation with $\alpha<\frac{1}{2}$, and let $\Tilde{L}(t)$ be given by
    \begin{equation}
    \Tilde{L}(t)=\sum_{n=1}^\infty \frac{\psi_n(t)}{n}.
    \end{equation}
     Then for all $0<t<T_{max}$,
     \begin{equation}
     \frac{\diff \Tilde{L}}{\diff t}
     \geq 
     -\left(\sum_{n=1}^\infty 
    \frac{1}{n^{2(1-\alpha)}}\right)^\frac{1}{2}
     \nu\|\psi\|_{\dot{H}^\alpha}
     +\frac{3}{\pi^2} \Tilde{L}^2.
     \end{equation}
\end{proposition}

\begin{proof}
    We will prove this by finding a lower bound on the nonlinearity and an upper bound on the dissipation.
    Begin by applying H\"older's inequality, finding that
    \begin{equation}
    \sum_{n=1}^\infty
    \frac{\psi_n}{n}
    \leq 
    \left(\sum_{n=1}^\infty \psi_n^2\right)^\frac{1}{2}
    \left(\sum_{n=1}^\infty \frac{1}{n^2}\right)^\frac{1}{2},
    \end{equation}
    and therefore
    \begin{equation}
    \left(\sum_{n=1}^\infty
    \frac{\psi_n}{n}\right)^2
    \leq 
    \frac{\pi^2}{6} \sum_{n=1}^\infty \psi_n^2,
    \end{equation}
    which implies that 
    \begin{equation}
    \frac{1}{2}\sum_{n=1}^\infty \psi_n^2
    \geq \frac{3}{\pi^2} \Tilde{L}^2.
    \end{equation}

    Next observe that 
    \begin{align}
    \sum_{n=1}^\infty \frac{\psi_n}{n^{1-2\alpha}}
    &=
    \sum_{n=1}^\infty 
    \frac{1}{n^{1-\alpha}} n^\alpha \psi_n \\
    &\leq 
    \left(\sum_{n=1}^\infty
    \frac{1}{n^{2(1-\alpha)}}\right)^\frac{1}{2}
    \left(\sum_{n=1}^\infty
    n^{2\alpha} \psi_n^2\right)^\frac{1}{2},
    \end{align}
    and this completes the proof.
\end{proof}

\begin{remark}
    Note that changing the order of summation for the Fourier coefficients in \Cref{FourierLyapunovProp} is equivalent to the integration by parts in \Cref{InnerProductLemma}. In both cases we need enough regularity for the result to apply, because the change to the order of summation is only legitimate when the Fourier coefficients are absolutely summable, which they must be due to higher regularity.
\end{remark}

\section{The general family of attractors} \label{VeryGeneralSection}

In this section, we will prove \Cref{VeryGeneralAttractorThmIntro,VeryGeneralBlowupThmIntro}, showing that not only $F$, but a whole range of functions sharing similar properties act as $L^2$ attractors for odd solutions of the inviscid Burgers equation (and as Lyapunov functionals in the viscous case). First we generalize \Cref{InnerProductLemma} to more general functions $H$.

\begin{lemma} \label{GeneralInnerProductLemma}
     Suppose that $H\in L^\infty\left(\mathbb{T}\right)$ is odd, differentiable on $(0,\pi)$, and that 
    \begin{equation}
    m:= \inf_{0<x<\pi} H'(x)
    >0.
    \end{equation}
    Then for all $u\in C^1\left(\mathbb{T}\right)$,
    \begin{equation}
    -\left<H,\frac{1}{2}\partial_x u^2\right>
    \geq 
    \frac{m}{2} \|u\|_{L^2}^2.
    \end{equation}
\end{lemma}

\begin{proof}
    Observe that $\frac{1}{2}\partial_x u^2$ is also odd, and therefore
    \begin{align}
    -\left<H,\frac{1}{2}\partial_x u^2\right>
    &=
    -\int_0^\pi H(x) \partial_x u^2(x) \\
    &=
    -\int_0^\pi H'(x) u^2(x) \diff x \\
    &\geq 
    m \int_0^\pi u^2(x) \diff x \\
    &=
    \frac{m}{2} \|u\|_{L^2}^2.
    \end{align}
    Note that the boundary terms vanish because $u$ is odd and continuous and $H$ is bounded, so 
    \begin{equation}
    \lim_{x\to 0}H(x)u^2(x)=
    \lim_{x\to \pi}H(x)u^2(x)
    =0.
    \end{equation}
\end{proof}

\begin{theorem} \label{VeryGeneralAttractorThm}
    Suppose that $H\in L^\infty\left(\mathbb{T}\right)$ is odd, differentiable on $(0,\pi)$, and that 
    \begin{equation}
    m:= \inf_{0<x<\pi} H'(x)
    >0.
    \end{equation}
    Suppose that $u\in C\left([0,T_{max});
    C^1\left(\mathbb{T}\right)\right)$
    is an odd solution of Burgers equation.
    Then for all $0\leq t<T_{max}$,
    \begin{equation}
    \|u(\cdot,t)-H\|_{L^2}^2
    \leq 
    \left\|u^0-H\right\|_{L^2}^2
    -mt.
    \end{equation}
    Note that this implies the solution blows up in finite-time with
    \begin{equation}
    T_{max}\leq\frac{\left\|u^0-H\right\|_{L^2}^2}{m}.
    \end{equation}
\end{theorem}

\begin{proof}
    First observe that due to conservation of energy,
    \begin{align}
    \|u(\cdot,t)-H\|_{L^2}^2
    &=
    \|u(\cdot,t)\|_{L^2}^2+\|H\|_{L^2}^2
    -2\left<H,u(\cdot,t)\right> \\
    &=
    \left\|u^0\right\|_{L^2}^2+\|H\|_{L^2}^2
    -2\left<H,u(\cdot,t)\right>.
    \end{align}
    Therefore we can apply \Cref{GeneralInnerProductLemma} to find that
    \begin{align}
    \frac{\diff}{\diff t}
    \|u(\cdot,t)-H\|_{L^2}^2
    &=
    -2\frac{\diff}{\diff t}
    \left<H,u(\cdot,t)\right> \\
    &=
    \left<H,\partial_x u^2\right> \\
    &\leq 
    -m\|u\|_{L^2}^2 \\
    &=
    -m\left\|u^0\right\|_{L^2}^2.
    \end{align}
    Integrating this differential inequality completes the proof.
\end{proof}

\begin{theorem}  
 Suppose that $H\in \dot{H}^\alpha\left(\mathbb{T}\right) \cap L^\infty\left(\mathbb{T}\right), 0<\alpha<\frac{1}{2}$ is odd, differentiable on $(0,\pi)$, and that 
    \begin{equation}
    m:= \inf_{0<x<\pi} H'(x)
    >0.
    \end{equation}    
    Suppose $u\in C\left([0,T_{max}),\dot{H}^s\right), s>\frac{3}{2}-2\alpha$ is an odd solution of the fractal Burgers equation with fractional dissipation $\alpha$ and with odd initial data satisfying
    \begin{equation}
    \Tilde{L}_0^3
    >\frac{12}{m}
    \|H\|_{\dot{H}^\alpha}^2
    \|H\|_{L^2}^2
    \left\|u^0\right\|_{L^2}^2\nu,
    \end{equation}
    where
    \begin{equation}
    \Tilde{L}_0
    =
    \int_{-\pi}^\pi
    H(x) u^0(x)\diff x.
    \end{equation}
Then the solution of the fractal Burgers equation blows up in finite-time
\begin{equation}
T_{max}<\frac{6\|H\|_{L^2}^2}{m\Tilde{L}_0}.
\end{equation}
\end{theorem}

\begin{proof}
    In this case we will take our Lyapunov functional to be
    \begin{equation}
    \Tilde{L}(t)=
    \int_{-\pi}^\pi H(x)u(x,t)\diff x.
    \end{equation}
    Applying \Cref{GeneralInnerProductLemma} and H\"older's inequality, we can see that
    \begin{align}
    \frac{\diff\Tilde{L}}{\diff t}
    &=
    -\nu\left<
    (-\Delta)^\frac{\alpha}{2}H,
    (-\Delta)^\frac{\alpha}{2}\right>
    -\left<H,\frac{1}{2}\partial_x u^2\right> \\
    &\geq 
    -\nu\|H\|_{\dot{H}^\alpha} 
    \|u\|_{\dot{H}^\alpha}
    +
    \frac{m}{2}\left\|u\right\|_{L^2}^2 \\
    &\geq 
    -\nu\|H\|_{\dot{H}^\alpha} 
    \|u\|_{\dot{H}^\alpha}
    +
    \frac{m}{2\|H\|_{L^2}^2} \Tilde{L}^2.
    \end{align}
    Let $f(t)$ be giving by
    \begin{equation}
    f(t)=\nu\|H\|_{\dot{H}^\alpha} 
    \|u(\cdot,t)\|_{\dot{H}^\alpha}.
    \end{equation}
    Observe that by applying H\"older's inequality to $1$ and $\|u(\cdot,t)\|_{H^\alpha}$ and making use of the energy equality, we can see that
    \begin{align}
    \int_0^t f(\tau)\diff \tau
    &\leq 
    \nu\|H\|_{\dot{H}^\alpha} 
    \left(\int_0^t 
    \|u(\cdot,\tau)\|_{\dot{H}^\alpha}^2 
    \diff \tau\right)^\frac{1}{2} 
    t^\frac{1}{2}\\
    &\leq 
    \nu\|H\|_{\dot{H}^\alpha} 
    \left( \frac{1}{2\nu}
    \left\|u^0\right\|_{L^2}^2 
    \diff \tau\right)^\frac{1}{2}
     t^\frac{1}{2} \\
    &=
    \frac{\|H\|_{\dot{H}^\alpha}\sqrt{\nu}}
    {\sqrt{2}} \left\|u^0\right\|_{L^2}  
    t^\frac{1}{2},
    \end{align}
    and that we have shown above that for all $0\leq t<T_{max}$,
    \begin{equation}
    \frac{\diff\Tilde{L}}{\diff t}
    \geq -f+
    \frac{m}{2\|H\|_{L^2}^2} \Tilde{L}^2.
    \end{equation}
    Applying \Cref{ODElemma} with
    \begin{align}
    M&= \frac{\|H\|_{\dot{H}^\alpha}}
    {\sqrt{2}} \left\|u^0\right\|_{L^2}
    \nu^\frac{1}{2} \\
    \kappa &= \frac{m}{2\|H\|_{L^2}^2},
    \end{align}
    completes the proof.
\end{proof}

\section*{Acknowledgments}

This material is based upon work supported by the Swedish Research Council under grant no. 2021-06594 while the author was in residence at the Institut Mittag-Leffler in Djursholm, Sweden during the Autumn 2023 semester.
The author would also like to thank Prof. Tarek Elgindi and Prof. Alexander Kiselev for their helpful comments and suggestions during his visit to Duke University, and Prof. Stephen Gustafson for the observation that arguments for the attractor $F$ could apply to a broader range of functions.
\bibliography{bib}

\end{document}